\documentclass[12pt]{article}
\usepackage[utf8]{inputenc}

\usepackage{amsmath,amsthm}
\usepackage{amssymb}
\usepackage{epsfig,tikz}
\usepackage{graphicx}
\usepackage{comment}
\usepackage{asymptote}
\usepackage{xypic}
\usepackage{hyperref}
\hypersetup{
    unicode=false,          
    pdftoolbar=true,        
    pdfmenubar=true,        
    pdffitwindow=false,     
    pdfstartview={FitH},    
    pdfnewwindow=true,      
    colorlinks=true,       
    linkcolor=blue,          
    citecolor=red,        
    filecolor=magenta,      
    urlcolor=cyan           
}

\usepackage[tmargin=1in,bmargin=1in,lmargin=1in,rmargin=1in]{geometry}

\newcommand{\Mbar}{\overline{M}}

\newcommand{\bA}{\mathbb{A}}
\newcommand{\bZ}{\mathbb{Z}}

\newcommand{\bC}{\mathbb{C}}

\newcommand{\bP}{\mathbb{P}}

\newcommand{\beq}{\begin{equation}}
\newcommand{\eeq}{\end{equation}}

\DeclareMathOperator{\Sym}{Sym}
\DeclareMathOperator{\PGL}{PGL}
\newcommand{\St}{\operatorname{St}}
\newcommand{\Comp}{\mathrm{Comp}}
\newcommand{\CPF}{\mathrm{CPF}}
\newcommand{\Good}{\mathrm{GPF}}
\newcommand{\Bad}{\mathrm{BPF}}
\newcommand{\Noncorner}{\mathrm{GPF}_\bullet}
\newcommand{\Corner}{\mathrm{BPF}_\bullet}
\newcommand{\ins}{\iota}
\newcommand{\rem}{\nu}

\newcommand{\dom}{d}
\newcommand{\rev}{\mathrm{rev}}

\newcommand{\defn}{\textbf}

\newtheorem{thm}{Theorem}
\newtheorem{lemma}[thm]{Lemma}
\newtheorem{prop}[thm]{Proposition}
\newtheorem{corollary}[thm]{Corollary}

\numberwithin{thm}{section}
\numberwithin{equation}{section}
\numberwithin{figure}{section}

\theoremstyle{definition}

\newtheorem{example}[thm]{Example}
\newtheorem{defi}[thm]{Definition}

\newtheorem{remark}[thm]{Remark}

\usepackage{mathtools}
\usepackage{enumitem}

\def\multichoose#1#2{\left<\genfrac{}{}{0pt}{}{#1}{#2}\right>}

\title{Projective Embeddings of $\Mbar_{0,n}$ and Parking Functions}
\author{Renzo Cavalieri, Maria Gillespie, Leonid Monin}
\date{\today}

\begin{document}

\maketitle{}
\begin{abstract}
    The moduli space $\Mbar_{0,n}$
    may be  embedded into the product of projective spaces $\bP^1\times \bP^2\times \cdots \times \bP^{n-3}$, using a combination of the Kapranov map $|\psi_n|:\Mbar_{0,n}\to \mathbb{P}^{n-3}$ and the forgetful maps $\pi_i:\Mbar_{0,i}\to \Mbar_{0,i-1}$.  We give an explicit combinatorial formula for the multidegree of this embedding in terms of certain parking functions of height $n-3$.  We use this combinatorial interpretation to show that the total degree of the embedding (thought of as the projectivization of its cone in $\bA^2\times \bA^3\cdots \times \bA^{n-2}$) is equal to $(2(n-3)-1)!!=(2n-7)(2n-9) \cdots(5)(3)(1)$.  As a consequence, we also obtain a new combinatorial interpretation for the odd double factorial.
\end{abstract}

\section{Introduction}

The moduli space of stable, $n$-marked, rational curves $\Mbar_{0,n}$ is a poster child for the field of combinatorial algebraic geometry.  It is a smooth,  projective variety and a fine and proper moduli space. It may be obtained from $\mathbb{P}^{n-3}$ by
a combinatorially prescribed sequence of blow-ups along smooth loci.
It is also a tropical compactification, meaning  that it can be realized as  the closure of a very affine variety inside a toric variety. The stratification induced by the boundary of the toric variety coincides with the natural stratification by homeomorphism classes of the objects parameterized; strata are  indexed by stable trees with $n$-marked leaves, and the graph algebra of stable trees completely controls the intersection theory of $\Mbar_{0,n}$, meaning that one may combinatorially define a multiplication on stable trees in such a way that the natural assignment of a tree with the (closure of the) stratum it indexes defines a surjective ring homomorphism to the Chow ring of $\Mbar_{0,n}$.  

This work provides another instance of the rich interaction between algebraic geometry and combinatorics brought about by $\Mbar_{0,n}$.  The starting point of this paper is the closed embedding $$\phi_n:\Mbar_{0,n}\to \bP^1\times \bP^2\times \cdots \times \bP^{n-3}$$ (defined in Corollary \ref{cor:embedding}) arising from  recent work of Keel and Tevelev \cite{KeT}.  
We study the degrees of the embedding from both a geometric and combinatorial perspective. We now state succinctly our two main results and then discuss them.

\begin{thm}\label{thm:main-1}
Let $n\geq{3}$  and $\mathbf{k} = \{k_1, \ldots, k_{n-3}\}$ be an ordered list of non-negative integers with $\sum k_i = n-3$. Then:
\begin{equation}\label{eq:stoeq}
    \deg_\mathbf{k}(\phi_n(\Mbar_{0,n})) \stackrel{Prop.\ref{prop:deg}}{=} 
    \int_{\Mbar_{0,n}} \prod_{i=1}^{n-3}{\omega_{3+i}^{k_i}} 
    \stackrel{Prop. \ref{prop:otoamc}}{=} 
  \multichoose{n-3}{\rev(\mathbf{k})}
  \stackrel{Thm. \ref{thm:CPF}}{=}
    |\CPF(n-3,\rev(\mathbf{k}))|.
\end{equation}
\end{thm}

\begin{thm}\label{thm:main-2}
    Denote by $C$ the affine cone over $\phi_n(\Mbar_{0,n})$ in $\bA^2\times \bA^{3}\times \cdots \bA^{n-2}$.  
    Then  \begin{equation}
    \deg(\mathbb{P}(C))=\sum_{\bf k} \deg_{\bf k}(\phi_n(\Mbar_{0,n}))=(2(n-3)-1)!!    
    \end{equation} where $(2n-7)!!=(2n-7)(2n-9) \cdots (5)(3)(1)$ is the odd double factorial. 
\end{thm}

In Theorem \ref{thm:main-1}, the first two quantities are geometric, the latter two are purely combinatorial.
The Chow ring of a product of projective spaces is generated by the (pull-backs via the projection functions of) hyperplane classes $H_i$ on each of the factors. The \textit{multidegree} of the embedding $\phi_n$ with respect to the tuple $\mathbf{k}=(k_1,\ldots,k_{n-3})\in \mathbb{N}^{n-3}$ is the coefficient of $\prod H_i^{i-k_i}$ in the expression of  $\phi_n(\Mbar_{0,n})$:
\begin{equation}
  [\phi_n(\Mbar_{0,n})] = \sum_\mathbf{k} \deg_\mathbf{k}(\phi_n(\Mbar_{0,n}))\prod H_i^{i-k_i} \in A^\ast(\mathbb{P}^1\times \ldots \times \mathbb{P}^{n-3}).
\end{equation}
By using Poincar\'e duality and the projection formula,   
 $\deg_\mathbf{k}(\phi_n(\Mbar_{0,n}))$ is described as an intersection number on $\Mbar_{0,n}$, the second term in the string of equalities in \eqref{eq:stoeq}.
 
 The \textit{omega class} $\omega_{i+3}$ corresponds to the hyperplane pullback $\phi_n^\ast(H_i)$. Since the building blocks of the embedding $\phi_n$ are the complete linear systems $|\psi_i|: \Mbar_{0,i}\to \mathbb{P}^{i-3}$, the class $\omega_i$ is the pull-back $f_i^\ast(\psi_i)$, where $f_i$ is the forgetful morphism forgetting all points with labels greater than $i$. 
 
 Intersection numbers arising from monomials in $\psi$ classes on $\Mbar_{0,n}$ are governed by the so called {\it string recursion} and  as a result are multinomial coefficients (\cite{kv:iqc}); 
 considering pullbacks of different $\psi$ classes via different forgetful morphisms breaks the $S_n$ symmetry and gives rise to an interesting recursive structure among intersection numbers of monomials of $\omega$ classes.     We define the symbol $\multichoose{m}{j_1,\ldots,j_n}$ to satisfy the corresponding recursion (see Definition \ref{def:asym} below), and obtain tautologically the second equality in \eqref{eq:stoeq},    where $\rev({\bf k})=(k_{n-3},k_{n-2},\ldots,k_1)$ is the tuple formed by reversing $\mathbf{k}$.

 Next, we show that these asymmetric analogs $\multichoose{n}{\bf k}$ of multinomial coefficients exhibit a remarkable combinatorial interpretation in terms of \textit{parking functions}.  Parking functions were first defined by Konheim and Weiss \cite{KW1966}, under the name of ``parking disciplines,'' as solutions to the following problem.  A parking lot with only one entrance on the west side has parking spaces numbered $1,2,\ldots,n$ in order from west to east.  Car number $1$ enters the lot first and drives to its preferred spot and parks there.  Each successive car $2,3,\ldots,n$ attempts to park in their preferred spot, but if it is taken, they keep driving until they find the next empty spot and park there.  For which sets of preferences do all cars end up parked?

  A \textit{parking function} is then defined to be a preference function $f$ on $\{1,2,\ldots,n\}$ mapping each car to its preferred spot, such that all cars end up parked.  It is known that $f$ is a parking function if and only if $$|f^{-1}(\{1,2,\ldots,i\})|\ge i$$ for all $i$.  Parking functions have become a central tool in many areas of recent research, perhaps most notably in the study of diagonal harmonics and $q,t$-analogs of Catalan numbers (see \cite{HaglundBook,HHLRU}). 
  
  For the last equality in Theorem \ref{thm:main-1}, we introduce the notion of \textit{column restriction} on parking functions (see Section \ref{sec:parking-functions}), and define $$\CPF(n,\mathbf{k})$$ to be the set of all column-restricted parking functions $f$ on $\{1,2,\ldots,n\}$ such that $k_{i}=|f^{-1}(i)|$ for all $i$. 
  
  The resulting combinatorial interpretation in Theorem \ref{thm:main-1} is the primary tool we use to prove Theorem \ref{thm:main-2}.  In particular, we show that the total number of column-restricted parking functions of size $n$ is $(2n-1)!!$, giving a new combinatorial interpretation of the double factorial. Theorem \ref{thm:main-2} was conjectured in \cite{MonRan}.

 We are aiming to communicate to an audience both of geometers and combinatorialists. For this reason, Section \ref{sec:bg} contains both geometric and combinatorial background, parts of which may be easily skipped by the expert readers. In Section \ref{sec:emb} we describe the embeddings $\phi_n$ and show that their multidegrees are computed as intersection numbers of $\omega$ classes on $\Mbar_{0,n}$. Section \ref{sec:om} develops the intersection theory of $\omega$ classes and introduces asymmetric multinomial coefficients as computing intersection numbers of monomials of $\omega$ classes. In Section \ref{sec:parking-functions} we make contact with the combinatorics of parking functions to show the last equality in Theorem \ref{thm:main-1} and to give two distinct (but similar) proofs of Theorem \ref{thm:main-2}.

\subsection{Acknowledgements}
This project started during the special program in combinatorial algebraic geometry in 2016. The authors are grateful for the stimulating environment provided by the
Fields Institute. R.C. is partially supported by Simons' collaboration grant 420720. L.M. is partially supported by the EPSRC Early Career Fellowship EP/R023379/1.

\section{Background}\label{sec:bg}
    \subsection{Geometry}

This section is aimed at collecting basic geometric background information and at establishing notation. The book \cite{fulton:it} is a comprehensive reference for Section \ref{subs:it}. An accessible and extensive introduction to the material in Section \ref{sec:m0n} is \cite{kv:iqc}. 

\subsubsection{Chow rings of products of projective spaces and multidegrees} 
\label{subs:it}

For a smooth algebraic variety $Y$, its \defn{Chow ring} $A^*(Y)$ is an algebraic version of De Rham cohomology. The elements of the $i$-th graded piece $A^i(Y)$ are integral linear combinations of irreducible subvarieties of $Y$ of codimension $i$ modulo rational equivalence. For two classes $Z_1\in A^i,\, Z_2\in A^j$, their product in $Z_1\cdot Z_2\in A^{i+j}(Y)$ is  the class of the intersection of transversely intersecting representatives of $Z_1, Z_2$.

For the product of projective spaces $\bP^{\bf b}=\bP^{b_1}\times\ldots\times \bP^{b_n}$, let $p_i:\bP^{\bf b}\to \bP^{b_i}$ be the natural projection on the $i$-th factor. We define the divisor classes $H_1,\ldots,H_n$ on $\bP^{\bf b}$ to be the pullbacks of hyperplanes in $\bP^{b_1},\ldots,\bP^{b_n}$ respectively:
\begin{equation}
  H_i:=p_i^* H_{\bP^{b_i}},  
\end{equation}
where $H_{\bP^{b_i}}$  is the class of a hyperplane on $\bP^{b_i}$.

The Chow ring $A^*(\bP^{\bf b})$ of $\bP^{\bf b}$ is generated by the classes  $H_1,\ldots, H_n$:
\begin{equation}
 A^*(\bP^{\bf b}) \simeq \bZ[H_1,\ldots,H_n]/\langle H_1^{b_1+1},\ldots,H_n^{b_n+1}\rangle.   
\end{equation}

\begin{defi}
  Let $Z\subset \bP^{\bf b}$ be a closed subvariety of the product of projective spaces. For any integer vector ${\bf k}=(k_1,\ldots, k_n)\in \bZ_{\geq 0}^n$ the {\bf degree of $Z$ of  index $\bf k$} is 
\begin{equation}
\deg_{\bf k}(Z) := \int_{\bP^{\bf b}} Z\cdot \prod_{i=1}^n H_i^{k_i},  
\end{equation}
where, in analogy with De Rham cohomology, we use integral notation to denote  the degree of the $0$-dimensional part of a cycle.
The collection of degrees of index ${\bf k}$, for all ${\bf k}\in \bZ_{\geq 0}^n,$ is called the {\bf multidegree} of $Z$.
\end{defi}
 If the dimension of $Z\cdot \prod_{i=1}^n H_i^{k_i}$ is nonzero, then by definition $\deg_{\bf k}(Z)=0$. Hence the degree of index $\bf k$ of $Z$ may be non-zero only  when $\sum k_i=\dim(Z)$. By Poincar\'e duality,  the class of $Z$ in the Chow ring $A^*(\bP^{\bf b})$ is determined by its multidegree:
$$
[Z]= \sum_{{\bf k},\, |{\bf k}|=\dim(Z)} \deg_{\bf k}(Z) \cdot H_1^{b_1-k_1}\ldots H_n^{b_n-k_n} \in A^{|{\bf b}|-\dim(Z)}(\bP^{\bf b}).
$$
If $\phi:X\to \bP^{\bf b}$ is a closed embedding,  by the projection formula (\cite{fulton:it}, Proposition 2.5) the degree of index $\bf k$ of the image $\phi(X)$ is equal to:
\begin{equation}
\deg_{\bf k} (\phi(X)) =  \int_X \prod_{i=1}^n (\phi^*H_i^{k_i}).   
\end{equation}

For $Z\subseteq  \bP^{\bf b}$ a closed subvariety, let $Con(Z)\subseteq \bA^{b_1+1}\times\ldots\times\bA^{b_n+1}$ be the affine cone over $Z$.
The following theorem of Van Der Waerden relates the multidegrees of $X$ with the degree of the projectivization $\mathbb{P}(Con(Z))$.
\begin{thm}[\cite{van}]
The degree $\deg(\mathbb{P}(Con(Z)))$ is equal to the sum of all multidegrees of $X$:
\begin{equation}
 \deg(\mathbb{P}(Con(Z)))= \sum_{{\bf k},\, |{\bf k}|=\dim(Z)}\deg_{\bf k}(Z).   
\end{equation}
\end{thm}

\subsubsection{The moduli space $\Mbar_{0,n}$ and its intersection theory}\label{sec:m0n}

For $n \geq 3$, the moduli space $M_{0,n}$ parameterizes ordered $n$-tuples of distinct points on $\bP^1$. We say that two $n$-tuples $(p_1,\ldots,p_n)$ and $(q_1,\ldots, q_n)$ are equivalent if there exists a projective transformation $g \in \PGL(2,\bC)$ such that:
$$
(q_1,\ldots, q_n) = (g(p_1),\ldots,g(p_n)).
$$ 
Since a projective transformation can map three chosen points on $\bP^1$ to any other three points and is uniquely determined by their image,  the dimension of  $M_{0,n}$ equals  $n-3$.

The space $M_{0,n}$ is not compact. Intuitively, this is because the points $p_i$ must all be distinct. There are a number of compactifications of $M_{0,n}$, including those described by Losev-Manin~\cite{losevmanin2000} and Keel~\cite{keel1992}. But the first and most well-known is $\Mbar_{0,n}$, the Deligne-Mumford compactification described explicitly by Kapranov~\cite{Ka2,Ka1}.

The moduli space~$\Mbar_{0,n}$ parametrizes families of \emph{stable $n$-pointed rational curves}.
\begin{defi}
 A {\bf stable rational $n$-pointed curve} is a tuple $(C, p_1, \ldots, p_n)$, where:
 \begin{enumerate}
 \item $C$ is a connected curve of arithmetic genus $0$ with at worst simple nodal singularities;
 \item $p_1, \ldots, p_n$ are distinct nonsingular points on $C$;
 \item each irreducible component of $C$ has at least three special points (either marked points or nodes).
 \end{enumerate}
\end{defi}

For the stable curve $(C, p_1, \ldots, p_n)$ we define its dual graph to have a vertex for each irreducible component of $C$, an edge between two vertices for each node between corresponding components, and a labeled half-edge for each marked point adjacent to the appropriate vertex. For $C$ to have arithmetic genus 0 the dual graph must be a tree.

The boundary $\Mbar_{0,n} \smallsetminus M_{0,n}$ is a simple normal crossing divisor. Intuitively this means that the irreducible components of the boundary locally intersect as coordinate hyperplanes in $\bC^n$. The boundary of $\Mbar_{0,n}$ has a natural stratification indexed by the dual graphs. 

The codimension of the stratum $\delta(\Gamma)$ in $\Mbar_{0,n}$ corresponding to the dual graph $\Gamma$ equals the number of edges of $\Gamma$.
Therefore, the irreducible divisorial components of the boundary of $\Mbar_{0,n}$ are given by dual graphs with one edge and  two vertices. These graphs are indexed by  partitions of the set $\{1,\ldots,n\}$ into two subsets $I$ and $I^c$, each of cardinality at least 2. We denote the corresponding irreducible boundary divisor of $\Mbar_{0,n}$ by $\delta_I$ (or equivalently by $\delta_{I^c}$). See Figure~\ref{m05} for an example of boundary divisors on $\Mbar_{0,5}$.

\begin{figure}
\begin{center}
\begin{tikzpicture}[line cap=round,line join=round,x=1.3cm,y=1.3cm]
\clip(-3.213019406818031,0.8782917403562502) rectangle (7.384083619169197,3.2810012383276104);
\draw [line width=1.pt] (-3.02,1.62)-- (0.08,2.84);
\draw [line width=1.pt] (2.30086907120197,1.6351813339818149)-- (-0.8060421312158056,2.83747197284204);
\draw (0.2,2.3) node[anchor=north west] {$p_i$};
\draw (1.13,1.95) node[anchor=north west] {$p_j$};
\draw (-2.4261058157793753,1.917017680527275) node[anchor=north west] {$p_k$};
\draw (-1.754606218093056,2.1793222108734933) node[anchor=north west] {$p_l$};
\draw (-1.104090982834434,2.4416267412197117) node[anchor=north west] {$p_m$};
\draw (-1.1,1.6652053313949053) node[anchor=north west] {$\delta_{\{i,j\}}=\delta_{\{k,l,m\}}$};
\draw [line width=1.pt] (3.9216621479755296,2.0839374497107896)-- (6.49225214797553,2.0839374497107896);
\draw [line width=1.pt] (3.9216621479755296,2.0839374497107896)-- (6.49225214797553,2.0839374497107896);
\draw [line width=1.pt] (3.9216621479755296,2.0839374497107896)-- (6.49225214797553,2.0839374497107896);
\draw [line width=1.pt] (3.9216621479755296,2.0839374497107896)-- (3.460007845189768,2.4940876472889553);
\draw [line width=1.pt] (3.9216621479755296,2.0839374497107896)-- (3.460007845189768,1.673787252132624);
\draw [line width=1.pt] (3.9216621479755296,2.0839374497107896)-- (3.35509,2.08394);
\draw [line width=1.pt] (3.35509,2.08394)-- (6.49225214797553,2.0839374497107896);
\draw [line width=1.pt] (6.49225214797553,2.0839374497107896)-- (7.037841639112188,2.336704929081224);
\draw [line width=1.pt] (6.49225214797553,2.0839374497107896)-- (7.037841639112188,1.8311699703403552);
\draw (6.744060565124424,2.651470365496686) node[anchor=north west] {$i$};
\draw (6.869966739690609,2.3157205666535265) node[anchor=north west] {$j$};
\draw (3.3,2.5) node[anchor=north west] {$k$};
\draw (3.3,2.16) node[anchor=north west] {$l$};
\draw (3.45,1.75) node[anchor=north west] {$m$};
\draw (4.5,1.6022522441118128) node[anchor=north west] {$\Gamma_{\{i,j\}}=\Gamma_{\{k,l,m\}}$};
\begin{scriptsize}
\draw [fill=black] (-0.36252777544152126,2.6658439077294656) circle (2.0pt);
\draw [fill=black] (-2.461425454413705,1.8398261114888) circle (2.0pt);
\draw [fill=black] (-1.8037498083866736,2.0986533012155673) circle (2.0pt);
\draw [fill=black] (-1.1344761051912589,2.362044887634408) circle (2.0pt);
\draw [fill=black] (0.5146934881829066,2.3263829947030334) circle (2.0pt);
\draw [fill=black] (1.4503565896132027,1.9643066704420922) circle (2.0pt);
\draw [fill=black] (3.9216621479755296,2.0839374497107896) circle (3.0pt);
\draw [fill=black] (6.49225214797553,2.0839374497107896) circle (3.0pt);
\end{scriptsize}
\end{tikzpicture}
\caption{$\Mbar_{0,5}$ has $\binom{5}{2}=10$ boundary divisors $\delta_{\{i,j\}}$, with $\{i,j\}\subset [5]$.}\label{m05}
\end{center}
\end{figure}
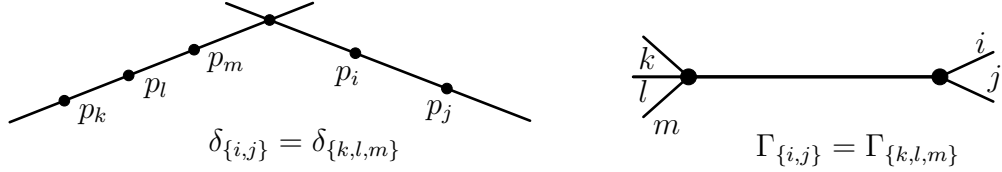
There are natural forgetful morphisms
\begin{equation}
    \pi_{n+1}:\Mbar_{0,n+1}\to \Mbar_{0,n},
\end{equation}
defined by forgetting the point marked $p_{n+1}$ and stabilizing the resulting curve if necessary. The morphism $\pi_{n+1}$ also functions as a universal family for $\Mbar_{0,n}$.

For $1\leq i \leq n$, the $i$-th tautological section morphism
\begin{equation}
    s_i:\Mbar_{0,n}\to \Mbar_{0,n+1}
\end{equation}
assigns to a curve $(C, p_1, \ldots, p_n)$ the $(n+1)-$pointed curve obtained by replacing the mark $p_i$ by a node connecting to a new rational component hosting the marks $p_i, p_{n+1}$.

\begin{defi}\label{def:psi}
 The {\bf $i$-th cotangent line bundle} $\mathbb{L}_i \to \Mbar_{0,n}$ 
 is defined to be
 \begin{equation}
     \mathbb{L}_i:= s_i^\ast(\omega_{\pi_{n+1}}),
 \end{equation}
 where $\omega_{\pi_{n+1}}$ denotes the relative dualizing sheaf of the universal family. Define 
 \begin{equation}
     \psi_i=c_1(\mathbb{L}_i)
 \end{equation} to be the first Chern class of $\mathbb{L}_i$.
\end{defi}
 Informally, one may think of $\mathbb{L}_i$ as the line bundle whose fiber over a point $[C,p_1,\ldots,p_n]$ is the cotangent space of $C$ at the marked point $p_i$.
 
 One may show that (\cite{k:psi}, Corollary 1.2.7)
\begin{equation}\label{van}
    \psi_i \cdot \delta_{\{i,j\}} = 0.
\end{equation}

We have an important comparison between the class $\psi_i$ on $\Mbar_{0,n+1}$ and the pullback of the class $\psi_i$ on $\Mbar_{0,n}$ via the forgetful morphism $\pi_{n+1}$ (\cite{AC:calc}, Lemma 3.1):

\begin{equation} \label{pbrel}
 \psi_i = \pi_{n+1}^\ast (\psi_i) + \delta_{\{i, n+1\}}.   
\end{equation}

Iterated applications of \eqref{pbrel}, \eqref{van}, and the projection formula give intersection numbers of $\psi$ classes remarkable combinatorial structure. Let $\St:\Sym^d\to \Sym^{d-1}$ be the linear transformation of the space of polynomials defined on monomials as

$$
\St\left(\prod x_i^{k_i}\right)=\sum_{i,k_i\ne 0} x_i^{k_i-1}\prod_{j\ne i}x_j^{k_j},
$$
and extended by linearity. The notation $\St$ is given due to the \defn{String equation} which provides a recursive formula for the intersection numbers of $\psi$ classes (\cite{k:psi}, Section 1.4). If $\prod \psi_i^{k_i}$ is a monomial in $\psi$ classes with $k_a=0$ for some $1\leq a \leq n$, then
\begin{equation}\label{string}
    (\pi_a)_* \prod \psi_i^{k_i} = \St \left(\prod \psi_i^{k_i} \right).
\end{equation}

If $\sum k_i =n-3$, we have (\cite{k:psi}, Lemma 1.5.1)
\begin{equation}
    \int_{\Mbar_{0,n}}\prod \psi_i^{k_i} = \St^{n-3} \left(\prod x_i^{k_i}\right)= \binom{n-3}{k_1,\ldots,k_n}.
\end{equation}
Here, the notation $\binom{n-3}{k_1,\ldots,k_n}$ refers to the multinomial coefficient defined in the next section.

\subsection{Combinatorics}\label{sec:comb-background}

In this section we provide some combinatorial background and notation.
We first recall some classical facts about multinomial coefficients, which we will generalize to asymmetric versions that enumerate a set of parking functions.  A good introductory reference for parking functions is \cite{HaglundBook}.

\subsubsection{Multinomial coefficients }
A \textbf{weak composition} of $n$ is a tuple $(k_1,\ldots,k_j)$ of nonnegative integers such that $\sum_{i=1}^j k_i=n$.  We say that $j$ is the \textbf{length} of the composition, and we write $\Comp(n,j)$ to denote the set of all weak compositions of $n$ having length $j$.  We often simply write $\mathbf{k}$ in boldface to denote a composition $(k_1,\ldots,k_j)$ if the length is understood.

Let $\mathbf{k}\in \Comp(n,j)$.  The \defn{multinomial coefficient} $\binom{ n}{\mathbf{k}}=\binom{n}{k_1,\ldots,k_j}$ is the coefficient of $x_1^{k_1}\cdots x_j^{k_j}$ in the expansion of $$(x_1+\cdots+x_j)^n.$$ This naturally generalizes the notion of a binomial coefficient.  It is well-known that the multinomial coefficients satisfy the explicit formula $$\binom{n}{k_1,\ldots,k_j}=\frac{n!}{k_1!k_2!\cdots k_j!}$$ and the recursion \begin{equation}\label{eq:recm}
    \binom{n}{k_1,\ldots,k_j}=\sum_{i=1}^{j} \binom{n-1}{k_1,\ldots,k_{i-1},k_{i}-1,k_{i+1},\ldots,k_j},
    \end{equation}
    where we define a multinomial coefficient $\binom{n}{\bf k}$ to be $0$ if any of the parts $k_i$ are negative.

In fact, the multinomial coefficients may be defined by the recursion (\ref{eq:recm}) along with the initial conditions $\binom{0}{0,0,\ldots,0}=1$.  Notice that the operator $\St$ from Section \ref{sec:m0n} is a reformulation of recursion \eqref{eq:recm}.

\subsubsection{Parking functions and Catalan compositions}

We will be primarily interested in compositions in $\Comp(n,n)$ with the following property.

\begin{defi}
   A composition $\mathbf{k}=(k_1,\ldots,k_n)\in \Comp(n,n)$ is \defn{Catalan} if for all $j<n$, we have $$k_1+k_2+\cdots+k_j\ge j.$$  
\end{defi}

A \defn{Dyck path} of height $n$ is a path in the first quadrant of the plane from $(0,0)$ to $(n,n)$, using only up or right steps of length $1$, which always stays weakly above the diagonal line $y=x$.  Notice that Dyck paths of height $n$ are naturally in bijection with Catalan compositions of length $n$, by setting $k_j$ to be the number of up-steps taken on the line $x=j-1$.  It is well-known that the number of Dyck paths of height $n$ is the Catalan number $C_n=\frac{1}{n+1}\binom{2n}{n}$, and hence $C_n$ is also the number of Catalan compositions of length $n$. 

A \defn{parking function} is a Dyck path along with a labeling of all unit squares having an up-step to its left with the numbers $1,2,\ldots,n$ in some order, such that in each column the numbers are increasing from bottom to top.  An example of a parking function for $n=6$ is shown in Figure \ref{fig:parking-function}.

\begin{figure}
	\begin{center}
		\includegraphics{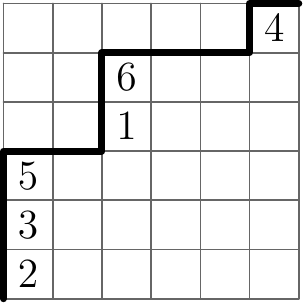}
	\end{center}
	
	\caption{\label{fig:parking-function} A parking function of size $6$.}
	
\end{figure}

Notice that a parking function may be specified by the sets of entries in each column from left to right.  The columns above are $\{2,3,5\},\{\},\{1,6\},\{\},\{\},\{4\}$.  Given this sequence, we can reconstruct the parking function by placing the column entries in increasing order in each column, with one entry per row going from bottom to top.  Then we draw the southeast-most path that lies northwest of the column labels to obtain the Dyck path.

Notice that the resulting path is a Dyck path, giving a valid parking function, if and only if the \defn{sequence of column heights} is Catalan.  In this example, the sequence of column heights is the composition $(3,0,2,0,0,1)$.

\begin{remark}
  The definition of parking function given here is equivalent to the historical definition in terms of parking cars given in the introduction.  In particular, we may think of the numbers in the columns as being the cars and the column that they reside in as their preferred spot.  In the above diagram, cars $2,3,5$ prefer spot number $1$, cars $1,6$ prefer spot $3$, and car $4$ prefers spot $6$.
\end{remark}

\section{Embeddings of $\Mbar_{0,n}$ in products of projective spaces}    \label{sec:emb}

In this section we describe an embedding $\phi_n:\Mbar_{0,n}\hookrightarrow\bP^1\times\cdots\times\bP^{n-3}$, first  obtained in \cite{KeT}. The embedding $\phi_n$ depends on two well-studied maps from $\Mbar_{0,n}$, namely the forgetful map $\pi_n:\Mbar_{0,n}\to\Mbar_{0,n-1}$ and Kapranov's map $|\psi_n|:\Mbar_{0,n}\to\bP^{n-3}$. 

The forgetful map $\pi_n:\Mbar_{0,n}\to\Mbar_{0,n-1}$ is the morphism given by forgetting the last point of $[C,p_1,\ldots p_n]$  and stabilizing the curve, i.e. contracting the components which have less than three special points and remembering the points of intersection.

 The Kapranov map  is  given by the linear system $|\psi_n|$, where $\psi_n$ is the first Chern class of the $n$-th cotangent line bundle from Definition \ref{def:psi}. This map is first described in detail by Kapranov~\cite{Ka1}: he shows that $\Mbar_{0,n}$, identified with the universal family over $\Mbar_{0,n-1}$, corresponds to the family $\mathcal{U}\subseteq \Mbar_{0,n-1}\times \mathbb{P}^{n-3}$ of rational curves through $(n-1)$ points in general position in $\mathbb{P}^{n-3}$. The cotangent line bundle $\mathbb{L}_n$ is identified with $\rho_2^\ast(\mathcal{O}_{\mathbb{P}^{n-3}}(1))$, where $\rho_2:\Mbar_{0,n-1}\times \bP^{n-3}\to \bP^{n-3}$ denotes the projection onto the second factor; this implies in particular that $|\psi_n| = \rho_2:\Mbar_{0,n}\to\bP^{n-3}$.

\begin{thm}\label{thm:embedding}\cite[Cor 2.7]{KeT} The map $\Phi_n=(\pi_n,\psi_n):\Mbar_{0,n}\rightarrow \Mbar_{0,n-1}\times \bP^{n-3}$ is a closed embedding.
\end{thm}
By applying Theorem~\ref{thm:embedding} iteratively, one obtains the following corollary.
\begin{corollary}\label{cor:embedding} We have a closed embedding $\phi_n:\Mbar_{0,n}\hookrightarrow \bP^1\times\bP^2\times\ldots \times\bP^{n-3}$.
\end{corollary}

\begin{proof}
Observe the commutative diagram \eqref{bigdiag}; all vertical arrows are natural projection functions, and any unlabeled horizontal arrow is the product of the unique labeled horizontal arrow below and the identity function on the remaining factor. For any $n\geq 4$, the embedding $\phi_n$ is obtained as the composition of $ \Phi_n$ with all horizontal arrows following it. The map $\phi_4= \Phi_4$ is an isomorphism.
\end{proof} 
\begin{equation}\label{bigdiag}
{\tiny
\xymatrix{
\Mbar_{0,n} \ar[r]^{\hspace{-1cm}\Phi_n}\ar[dr]_{\hspace{.8cm} f_n=\pi_n} \ar@/_3pc/[ddrr]_{f_{n-1}}\ar@/_6pc/[ddddrrrr]_{f_{4}}
&
\Mbar_{0,n-1}\times \mathbb{P}^{n-3} \ar[r]\ar[d]
&
\Mbar_{0,n-2}\times \mathbb{P}^{n-4}\times  \mathbb{P}^{n-3} \ar[r]\ar[d]
&
\ldots \ar[r]\ar[d]
&
\Mbar_{0,4}\times \mathbb{P}^{2} \times \ldots \times   \mathbb{P}^{n-3} \ar[r]\ar[d]
&
\mathbb{P}^{1}  \times \ldots \times  \mathbb{P}^{n-3}\ar[d]
\\
&
\Mbar_{0,n-1} \ar[r]^{\hspace{-1cm}\Phi_{n-1}}\ar[dr]_{\pi_{n-1}}
&
\Mbar_{0,n-2}\times \mathbb{P}^{n-4} \ar[r]\ar[d]
&
\ldots \ar[r]\ar[d]
&
\Mbar_{0,4}\times \mathbb{P}^{2} \times \ldots \times   \mathbb{P}^{n-4} \ar[r]\ar[d]
&
\mathbb{P}^{1}  \times \ldots \times  \mathbb{P}^{n-4}\ar[d]
\\
&
&
\Mbar_{0,n-2} \ar[r]^{\hspace{-.1cm}\Phi_{n-2}}\ar[dr]_{\pi_{n-2}}
&
\ldots \ar[r]\ar[d]
&
\Mbar_{0,4}\times \mathbb{P}^{2} \times \ldots \times   \mathbb{P}^{n-5} \ar[r]\ar[d]
&
\mathbb{P}^{1}  \times \ldots \times  \mathbb{P}^{n-5}\ar[d]\\
&&&\ldots\ar[dr]_{\pi_5}&\ldots\ar[d]&\ldots\ar[d]\\
&&&&\Mbar_{0,4}\ar[r]^{\hspace{-.1cm}\Phi_4 } & \mathbb{P}^1.
}}
\end{equation}

Let  $H_1,\ldots,H_{n-3}$ as before be the pullbacks of hyperplane classes in $\bP^{1},\ldots,\bP^{n-3}$ respectively. Let $f_i: \Mbar_{0,n} \to \Mbar_{0,i}$ be the forgetful map which is forgetting points labelled by $i+1,\ldots,n$, i.e.
\begin{equation}
f_i = \pi_{i+1}\circ\cdots\circ\pi_n.
\end{equation}

By commutativity of \eqref{bigdiag}, the class  $\phi_n^*H_i$ on $\Mbar_{0,n}$ is equal to $f_{i+3}^* \psi_{i+3}$, where $\psi_{i+3}$ is understood as a $\psi$ class on $\Mbar_{0,i+3}$. Motivated by this fact we introduce the following notation.

\begin{defi}\label{def:omega}
We define the \defn{omega class} $\omega_i:=f_i^*(\psi_i)$ on $\Mbar_{0,n}$ to be the pullback of the corresponding $\psi$ class from $\Mbar_{0,i}$. 
\end{defi}

Applying this notation to the discussion following diagram \eqref{bigdiag} one obtains the following proposition.
\begin{prop}\label{prop:deg}
The degree of $\phi_n(\Mbar_{0,n})$ of index ${\bf k}=(k_1,\ldots,k_{n-3})$ is nonzero only if $\sum k_i =n-3$ and is equal to:
\begin{equation}
 \deg_\mathbf{k}(\phi_n(\Mbar_{0,n}))= \int_{\Mbar_{0,n} } \prod_{i=1}^{n-3} \omega_{i+3}^{k_{i}}.   
\end{equation}
\end{prop}
\begin{proof}
We have the following string of equalities
$$  \deg_\mathbf{k}(\phi_{n}(\Mbar_{0,n}))=\int_{\bP^1\times \ldots \times\mathbb{P}^{n-3}}{\phi_{n}}_\ast([\Mbar_{0,n}]) \prod_{i=1}^{n-3} H_i^{k_i} = \int_{\Mbar_{0,n} } \prod_{i=1}^{n-3} \phi_n^\ast(H_i^{k_i}) =  \int_{\Mbar_{0,n} } \prod_{i=1}^{n-3} \omega_{i+3}^{k_{i}},  $$  
where the first equality is the definition of multidegree, the second is given by the projection formula (\cite{fulton:it}, Proposition 2.5), and the third follows from the commutativity of diagram \eqref{bigdiag} and Definition \ref{def:omega}.
\end{proof}

\section{Intersections of $\omega$ classes and asymmetric multinomials}\label{sec:om}

In this section we obtain some results on the intersection theory of omega classes; in particular we show that top intersection numbers satisfy a recursion that leads us to define a notion of asymmetric multinomial coefficients.

For a permutation $\sigma\in \Sym_n$, let  $P_\sigma: \Mbar_{0,n} \to \Mbar_{0,n}$ be an automorphism which changes the order of marked points on $\Mbar_{0,n}$, i.e.
$$
P_\sigma: (C, p_1, \ldots, p_n) \mapsto (C, p_{\sigma(1)}, \ldots, p_{\sigma(n)}).
$$

\begin{lemma}\label{lem:invar}
The class of a monomial $M=\omega_1^{k_{1}}\ldots\omega_{a-1}^{k_{a-1}}$ in $\omega$ classes of indices $<a$ is invariant under permutation of indices $\geq a$. More precisely, for any $M\in A^*(\Mbar_{0,n})$ as before and $\sigma \in \Sym([a,n])$,
$$
P_\sigma^{\ast}(M)=M.
$$
\end{lemma}

\begin{proof}
The Lemma follows from the fact that $M=\omega_1^{k_{1}}\ldots\omega_{a-1}^{k_{a-1}}$ is a pull-back of the same monomial from $\Mbar_{0,a-1}$ under the forgetful map $f_{a-1}$, which is invariant under the action of $P_\sigma^{\ast}$.
\end{proof}

\begin{lemma}\label{cor:symmetric}
  Let $M=\omega_1^{k_{1}}\ldots\omega_{a-1}^{k_{a-1}}$ and $\sigma \in \Sym([a,n])$ be as before. Then for any monomial in $\psi$ classes of the form $N=\psi_{a}^{k_a}\ldots\psi_n^{k_n}$, we have
  \begin{equation}\label{eq:lem3}
   \int_{\Mbar_{0,n}}M\cdot N = \int_{\Mbar_{0,n}} M\cdot P_\sigma^*(N) = \int_{\Mbar_{0,n}} M\cdot\psi_{\sigma^{-1}(a)}^{k_a}\ldots\psi_{\sigma^{-1}(n)}^{k_n}.    
  \end{equation}
\end{lemma}
\begin{proof}
The above expression is equal to $0$ unless $\deg(M)+\deg(N)=n-3$. If the product $M\cdot N$ is of top degree, then $M\cdot N= P^*_\sigma(M\cdot N)=P^*_\sigma(M)\cdot P^*_\sigma(N)$. But by Lemma~\ref{lem:invar}, $P^*_\sigma(M)=M$, so \eqref{eq:lem3} is proved.
\end{proof}

\begin{lemma}\label{lem:otop}
   Let  $ \prod_{j=a+1}^{n} \omega_j^{k_{j}}$ be a monomial in omega classes such that $k_j >0$ for all $a< j\leq n$. Then
   \begin{equation}
 \prod_{j=a+1}^{n} \omega_j^{k_{j}}
   =
  \prod_{j=a+1}^{n} \psi_j^{k_{j}}. 
   \end{equation}
\end{lemma}
\begin{proof}
    We prove the statement by descending induction on $a$. The base case $a=n-1$ is true  since $\omega_n = \psi_n$ by definition. Assume by induction that $ \prod_{j=a+2}^{n} \omega_j^{k_{j}}
   =
  \prod_{j=a+2}^{n} \psi_j^{k_{j}}$ whenever all exponents are positive, and consider a monomial 
    $
    \omega_{a+1}^{k_{a+1}}\prod_{j=a+2}^n\psi_j^{k_j},
    $ with $k_{a+1}>0$.
   Using equation \eqref{pbrel} to pull-back $\omega_{a+1}^{k_{a+1}}$ via $\pi_{a+2}$ and  using that $f_{a+1} = \pi_{a+2}\circ f_{a+2}$, one has
    \begin{align*}
      \omega_{a+1}^{k_{a+1}}\prod_{j=a+2}^n\omega_j^{k_j}&= f_{a+2}^\ast(\pi_{a+2}^\ast(\psi_{a+1}))^{k_{a+1}}\prod_{j=a+2}^n\omega_j^{k_j}
      \\
      &=
     f_{a+2}^\ast\left(\left(\psi_{a+1} - \delta_{\{a+1,a+2\}}\right)^{k_{a+1}}\psi_{a+2}^{k_{a+2}}\right) \prod_{j=a+3}^n\omega_j^{k_j}.
    \end{align*}
By \eqref{van}, $\delta_{\{a+1,a+2\}}\psi_{a+2} = 0$ and since $k_{a+2} >0$,
\begin{equation}
    \omega_{a+1}^{k_{a+1}}\prod_{j=a+2}^n\omega_j^{k_j}=
    f_{a+2}^\ast(\psi_{a+1})^{k_{a+1}} \prod_{j=a+2}^n\omega_j^{k_j}.
    \end{equation}
    We may repeat the same argument for the forgetful morphisms $\pi_{a+3}, \ldots, \pi_n$, and at each step $j = a+3, \ldots, n$ the boundary corrections $\delta_{\{a+1, j\}}$ are annihilated by the class $\psi_j^{k_j}$; in the end one  obtains 
    \begin{equation}
    \omega_{a+1}^{k_{a+1}}\prod_{j=a+2}^n\omega_j^{k_j}=
   \psi_{a+1}^{k_{a+1}} \prod_{j=a+2}^n\omega_j^{k_j} = \psi_{a+1}^{k_{a+1}} \prod_{j=a+2}^n\psi_j^{k_j},
    \end{equation}
 where the last equality is obtained by applying the inductive hypothesis.

\end{proof}

\begin{lemma}\label{lem:omega-classes}
Let  $ \prod_{i=1}^{a-1} \omega_i^{k_{i}} \prod_{j=a+1}^{n} \omega_j^{k_{j}}$ be a monomial in omega classes such that $k_j >0$ for all $a< j\leq n$. Then the following relation holds:
$$
(\pi_{a})_*\left( \prod_{i=1}^{a-1} \omega_i^{k_{i}} \prod_{j=a+1}^{n} \omega_j^{k_{j}}\right)=  \prod_{i=1}^{a-1} \omega_i^{k_{i}} \cdot \St \left(\prod_{j=a+1}^{n} \psi_j^{k_{j}}\right).
$$
\end{lemma}

\begin{proof}
From Lemma \ref{lem:otop} we may replace the $\omega$ classes with indices greater than $a$ by $\psi$ classes.

Next we notice that $\omega_i = \pi_a^* \omega_i$ for any $i< a$, where the $\omega_i$ on the right hand side still refers to the class obtained by pulling back $\psi_i$ via the forgetful morphism that forgets all marks greater than $i$ except for the mark $a$.
We therefore have $ \prod_{i=1}^{a-1} \omega_i^{k_{i}} = \pi_a^* \left(\prod_{i=1}^{a-1} \omega_i^{k_{i}}\right)$.
Applying the projection formula and the string equation~(\ref{string}), we obtain:

\begin{equation}\label{eq:sfo}
(\pi_{a})_*\left(\pi_a^* \left(\prod_{i=1}^{a-1} \omega_i^{k_{i}}\right)\cdot \prod_{j=a+1}^{n} \psi_j^{k_{j}} \right)=
 \prod_{i=1}^{a-1} \omega_i^{k_{i}}\cdot
(\pi_{a})_*\left(\prod_{j=a+1}^{n} \psi_j^{k_{j}}\right)  = \prod_{i=1}^{a-1} \omega_i^{k_{i}}\cdot
\St \left(\prod_{j=a+1}^{n} \psi_j^{k_{j}}\right) .    
\end{equation}
\end{proof}
At this point it is very tempting to claim a recursion among intersection numbers of $\omega$ classes  by applying Lemma \ref{lem:otop} in reverse and replacing $\psi$ classes back with $\omega$ classes in the last term of \eqref{eq:sfo}. However, a bit of care is needed, as not all terms of the expression $\St \left(\prod_{j=a+1}^{n} \psi_j^{k_{j}}\right)$ are guaranteed to have positive exponents for all $\psi_j$ with $j>a$, in particular if some $k_j=1$. Example \ref{ex:rec} gives a concrete illustration of how to address this subtlety and might be useful to refer to in order to navigate the proof of Lemma \ref{lem:sttil}.

\begin{lemma}\label{lem:sttil}
With notation as in Lemma \ref{lem:omega-classes}, we have
\begin{equation}\label{eq:ptoo}
 \St \left(\prod_{j=a+1}^{n} \psi_j^{k_{j}}\right) = \sum_{\{j\geq a+1| k_j \not=1\}} \omega_{j-1}^{k_j-1}\prod_{l\not=j} \omega_{l-1}^{k_{l}}
 +   
 \sum_{\{j\geq a+1| k_j =1\}}\prod_{l<j} \omega_{l}^{k_{l}}\prod_{l> j} \omega_{l-1}^{k_{l}}.
\end{equation}
\end{lemma}
\begin{proof}
We write the string recursion as:
\begin{equation}\label{eq:stst}
 \St \left(\prod_{j=a+1}^{n} \psi_j^{k_{j}}\right) = \sum_{\{j\geq a+1| k_j \not=1\}} \psi_{j}^{k_j-1}\prod_{l\not=j} \psi_{l}^{k_{l}}
 +   
 \sum_{\{j\geq a+1| k_j =1\}}\prod_{l<j} \psi_{l}^{k_{l}}\prod_{l> j} \psi_{l}^{k_{l}}.
\end{equation}
 For the first summand of \eqref{eq:stst},  we apply Lemma \ref{lem:otop}, taking care to reindex the $\omega$ classes since we have forgotten the mark labeled $a$; this yields the first summand of \eqref{eq:ptoo}. For any term of the second summand in \eqref{eq:stst}, we first apply Lemma \ref{cor:symmetric}: a cyclic permutation of the indices $a+1, \ldots, j$ will give in general a different intersection cycle, but it
 does not change the degree of any  intersection cycle obtained by multiplying by a monomial in omega classes with indices strictly less than $a$; we choose the permutation $(a+1 \ a+2 \ \ldots \ j-1 \ j )$ so that all exponents of $\psi_j$'s, with $j>a+1$, are strictly positive. Now Lemma \ref{lem:otop} may be applied; reindexing the omega classes because the mark labeled $a$ is no longer present, one obtains the second summand of \eqref{eq:ptoo}. 
\end{proof}
Combining the results of Lemmas \ref{lem:omega-classes} and \ref{lem:sttil} one obtains a recursion among intersection numbers of $\omega$ classes.
\begin{example}\label{ex:rec}
Consider 
\begin{equation}\label{eq:exint}
\int_{\Mbar_{0,10}} \omega_4\omega_8^2\omega_9\omega_{10}^3.
\end{equation}
We start by using Lemma \ref{lem:omega-classes}: since $7$ is the largest index $i$ of an $\omega_i$ that does not occur in the product, we push forward the integrand via $\pi_7$ and obtain
\begin{equation}\label{eq:appst}
(\pi_7)_\ast(\omega_4\omega_8^2\omega_9\omega_{10}^3) = \omega_4\St(\psi_8^2\psi_9\psi_{10}^3) = \omega_4(\psi_8\psi_9\psi_{10}^3+\psi_8^2\psi_{10}^3+\psi_8^2\psi_9\psi_{10}^2).
\end{equation}
The rightmost term in the string of equalities \eqref{eq:appst} is a zero-dimensional intersection cycle on a moduli space of nine-pointed rational curves, where the points are labeled $1,\ldots, \hat7, \ldots, 10$. Such a moduli space may be identified with $\Mbar_{0,9}$ simply by shifting down by one the indices of the last three marked points, to obtain that the intersection number in \eqref{eq:exint} is equal to
\begin{equation}\label{eq:almost}
 \int_{\Mbar_{0,9}}\omega_4(\psi_7\psi_8\psi_{9}^3+\psi_7^2\psi_{9}^3+\psi_7^2\psi_8\psi_{9}^2).   
\end{equation}
We now apply Lemma \ref{cor:symmetric} only to the second summand in \eqref{eq:almost}, and permute the indices $7$ and $8$, to obtain:
\begin{equation}\label{eq:almostthere}
 \int_{\Mbar_{0,9}}\omega_4(\psi_7\psi_8\psi_{9}^3+\psi_8^2\psi_{9}^3+\psi_7^2\psi_8\psi_{9}^2).   
\end{equation}
Finally we are in a position to apply Lemma \ref{lem:otop} and express $\eqref{eq:exint}$ in terms of a sum of intersection cycles of omega classes:
\begin{equation}\label{there}
 \int_{\Mbar_{0,10}} \omega_4\omega_8^2\omega_9\omega_{10}^3   = \int_{\Mbar_{0,9}}\omega_4(\omega_7\omega_8\omega_{9}^3+\omega_8^2\omega_{9}^3+\omega_7^2\omega_8\omega_{9}^2).
\end{equation}
We have expressed  the intersection cycle of omega classes \eqref{eq:exint}  on a space of ten-pointed curves as a sum of intersection cycles of omega classes on a space of nine-pointed curves. In order to combinatorialize this recursion, we look at the vector of exponents of the $\omega$  classes; we start at $n = 4$ since we know that $\omega_1 = \omega_2=\omega_3 = 0$. The vector of exponents for \eqref{eq:exint} is equal to $(1,0,0,0,2,1,3)$; the vectors of exponents for the three terms of the recursive expression in \eqref{there} are $(1,0,0,1,1,3), (1,0,0,0,2,3), (1,0,0,2,1,2)$.
\end{example}

\begin{remark}
 In computing the multidegrees of $\phi$, we only need to consider products of the $n-3$ omega classes $\omega_4,\ldots,\omega_n$.  In particular, for the multidegree $\deg_{(1,1,\ldots,1)}(\phi_n(\Mbar_{0,n}))$, we compute $\int_{\Mbar_{0,n}}\omega_4\omega_5\cdots\omega_n.$  There is therefore always a largest $\omega$ index that does not occur in the product.  In this case, forgetting the $3$rd marked point and applying Lemmas \ref{lem:omega-classes} and \ref{lem:sttil} gives us $$\int_{\Mbar_{0,n}}\omega_4\omega_5\cdots\omega_n=(n-3)\int_{\Mbar_{0,n-3}}\omega_4\cdots \omega_{n-1}=\cdots =(n-3)!.$$ 
\end{remark}

We proceed to give a general description of the recursion illustrated in Example \ref{ex:rec} with a language that makes contact with the combinatorial structure of parking functions. Unfortunately, that entails having to reverse the order of the vector of exponents for the monomial in $\omega$-classes.

\begin{defi}
	For a weak composition $\mathbf{k}\in \Comp(n,n)$, let $k_i$ be the leftmost $0$ in $k$, and let $j<i$ be a positive integer (where we set $i=n+1$ if there are no zeroes in $\mathbf{k}$).
	Then define $\widetilde{\mathbf{k}}_j$ to be the composition in $\Comp(n-1,n-1)$ formed by decreasing $k_j$ by $1$ and then removing the leftmost $0$ (which is either in position $j$ or $i$) from the resulting tuple. 
\end{defi}

For example, if $\mathbf{k}=(3,1,2,0,0,1,0)$, then $\widetilde{\mathbf{k}}_1=(2,1,2,0,1,0)$, $\widetilde{\mathbf{k}}_2=(3,2,0,0,1,0)$, and $\widetilde{\mathbf{k}}_3=(3,1,1,0,1,0)$.  Since $i=4$ in this example, $\widetilde{\mathbf{k}}_4$ is not defined.

\begin{defi}
  We  define the \defn{reverse} of a composition ${\bf k} = (k_1,\ldots,k_n)$, denoted $\rev(\bf k)$, to be the composition $$\rev({\bf k})=(k_n,k_{n-1},\ldots,k_1).$$ 
\end{defi}
 With these definitions in place, we can efficiently  describe the recursion on intersection numbers of $\omega$ classes. 

\begin{prop}\label{prop:rec}
For a weak composition $\mathbf{k}\in Comp(n-3,n-3)$, let $i_{\mathbf{k}}$ denote the index of the leftmost $0$. Denote by $\omega^{\rev(\mathbf{k})}$ the class $\omega_4^{k_{n-3}}\cdot \ldots \cdot\omega_n^{k_{1}}$.
Then:
\begin{equation}\label{eq:rrev}
 \int_{\Mbar_{0,n}} \omega^{\rev(\bf{k})}= \sum_{j=1}^{i_{\rev(\mathbf{k})}-1}\int_{\Mbar_{0,n-1}} \omega^{\rev(\widetilde{\mathbf{k}}_j)}.    
\end{equation} 
\end{prop}
\begin{proof}
  As illustrated in Example \ref{ex:rec}, a recursion for intersection numbers of omega classes may be described in terms of the exponent vectors of monomials $\prod \omega_i^{k_i}$ we wish to consider. Since $\omega_1, \omega_2, \omega_3 = 0$, we omit the first three $0$'s and consider the exponent vector $\mathbf{k} = (k_4, \ldots, k_n)$. We forget the highest labeled marked point $a$ such that $k_a = 0$, so in particular we are guaranteed that $k_i>0$ for all marks $i>a$. By Lemmas \ref{lem:omega-classes}, \ref{lem:sttil} one may express the intersection number $\omega^{\mathbf{k}}$ as a sum of intersection numbers corresponding to exponent vectors constructed as follows:
  \begin{itemize}
      \item the entries before $a$ are the same as the entries of $\mathbf{k}$;
      \item one obtains a summand for each $j>a$;
      \item if $k_j>1$, the new exponent vector is obtained by deleting the $0$ in position $a$ and replacing $k_j$ by $k_{j-1}$;
      \item if $k_j= 1$, the new exponent vector is obtained by deleting $k_j$.
  \end{itemize}
In order to connect this recursion with the combinatorics of parking functions, one has to reverse the order of the exponent vectors. Once one does that, the recursion just described becomes \eqref{eq:rrev}.
  \end{proof}

\begin{defi}\label{def:asym}
	Let $\mathbf{k}\in \Comp(n,n)$.  The \defn{asymmetric multinomial coefficients} $\multichoose{n}{\bf k}$ are defined by the recursion $\multichoose{1}{1}=1$ and \begin{equation}\multichoose{n}{\mathbf{k}}=\sum_{j=1}^{i_{\mathbf{k}}-1} \multichoose{n-1}{\widetilde{\mathbf{k}}_j},\end{equation} where $i_{\mathbf{k}}$ is the index of the leftmost $0$ in a composition $\mathbf{k}$.
\end{defi}

Proposition \ref{prop:rec} and Definition \ref{def:asym}, along with the initial conditions $$\int_{\Mbar_{0,4}}\omega_4 =  \int_{\Mbar_{0,4}}\psi_4 = 1,$$ lead to the following statement.
\begin{prop}\label{prop:otoamc}
For a weak composition $\mathbf{k}\in Comp(n-3,n-3)$,
\begin{equation}\label{eq:otoamc}
   \int_{\Mbar_{0,n}} \omega^{\bf{k}} =  \multichoose{n-3}{\rev(\mathbf{k})}. 
\end{equation}
\end{prop}
\begin{proof}
    The intersection cycle $\omega^{\mathbf{k}}$ can be computed recursively by reversing the order of its exponent vector, and then applying the same recursion  that defines the asymmetric multinomial coefficient  $\multichoose{n-3}{\rev(\mathbf{k})}$. Since the initial conditions for the two recursions agree,  $\int_{\Mbar_{0,4}}\omega_4  = 1  = \multichoose{1}{1}$, the result follows.
    \end{proof}

For the remainder of this section we explore some combinatorial properties of these asymmetric multinomial coefficients.

\begin{lemma}\label{lem:Catalan}
    If $\bf k\in \Comp(n,n)$ is Catalan, then $\widetilde{\bf k}_j$ is also Catalan for any $j\le i_{\bf k}$.  Conversely, if $\bf{k}$ is not Catalan then $\widetilde{\bf k}_j$ is not Catalan for any $j\le i_{\bf k}$.
\end{lemma}

\begin{proof}
    For $t<j$, the partial sum $k_1+\cdots+k_t\ge t$ is the same in $\widetilde{\bf k}_j$ as in $\bf k$.  
    
    First suppose $\bf k$ is Catalan.  Since $k_1+\cdots+k_{j-1}\ge j-1$, the partial sums then remain large enough before the first zero in $\widetilde{\bf k}_j$.  After removing this zero, the remaining partial sums only decreased by $1$ from $\bf k$ to  $\widetilde{\bf k}_j$, and their index has reduced by $1$ as well, so $\widetilde{\bf k}_j$ is Catalan.
    
    If $\bf k$ is not Catalan, then $k_1+\cdots+k_t<t$ for some $t$.  If $t<j$ we are done.  Otherwise, the $(t-1)$st partial sum in $\widetilde{\bf k}_j$ is one less than $k_1+\cdots+k_t$ and so it is strictly less than $t-1$.  Thus $\widetilde{\bf k}_j$ is also not Catalan.
\end{proof}

Lemma \ref{lem:Catalan} allows us to determine which of the asymmetric multinomial coefficients are nonzero.  In particular, the smallest Catalan coefficient $\multichoose{1}{1}=1$ is nonzero, and the smallest non-Catalan coefficient, $\multichoose{2}{0,2}$, is zero since the sum in the recursion of Definition \ref{def:asym} is empty for sequences starting with $0$.  We therefore obtain the following characterization from Lemma \ref{lem:Catalan} by a simple induction on $n$.

\begin{corollary}
    The coefficient $\multichoose{n}{\bf k}$ is nonzero if and only if $\bf k\in \Comp(n,n)$ is Catalan.
\end{corollary}

Finally, we note that the recursion in Definition \ref{def:asym} immediately gives rise to the following recursive formula for the generating function $$F_n(x_1,\ldots,x_n):=\sum_{\mathbf{k}\in \Comp(n,n)} \multichoose{n}{\bf k} x_1^{k_1}\cdots x_n^{k_n}.$$

Define $\Comp(n,n,i)$ to be the set of all compositions $\bf k\in \Comp(n,n)$ for which $i$ is the index of the first zero in $\bf k$, that is, $k_i=0$ and for all $j<i$, $k_j\neq 0$.  If no index of $\bf k$ is zero we say $i=n+1$ and write $\Comp(n,n,n+1)$ for the set of such parking functions.  We use the auxiliary generating functions $$F_{n,i}(x_1,\ldots,x_n)=\sum_{\mathbf{k}\in \Comp(n,n,i)}\multichoose{n}{\bf k}x_1^{k_1}\cdots x_n^{k_n}.$$  In order to simplify our notation, we write $X$ for the set of variables $x_1,\ldots,x_n$.  The notation $F(X\backslash x_i)$ means that we are plugging in $x_1,\ldots,x_{i-1},x_{i+1},\ldots,x_n$ into the function $F$.

\begin{prop}
  We have $$F_n(X)=\sum_{i=2}^{n+1} F_{n,i}(X)$$ where the functions $F_{n,i}(X)$ satisfy the recursion \begin{equation}\label{eq:gfrec}F_{n,i}(X)=\left(\sum_{j=1}^{i-1} x_jF_{n-1,i-1}(X\backslash x_j)\right)+(x_1+\cdots+x_{i-1}) \sum_{\ell=i}^{n}F_{n-1,\ell}(X\backslash x_i)\end{equation} with initial condition $F_{1,2}(x_1)=x_1$.
\end{prop}

\begin{proof}
The term $x_jF_{n-1,i-1}(X\backslash x_j)$ enumerates the compositions of the form $\multichoose{n-1}{\widetilde{\bf k}_j}$ in which $k_j=1$, and the term $x_j\sum_{\ell=i}^{n}F_{n-1,\ell}(X\backslash x_i)$ enumerates the compositions $\multichoose{n-1}{\widetilde{\bf k}_j}$ in which $k_j>1$.  Summing over all possible $j$ completes the proof.
\end{proof}
Summing  recursion \eqref{eq:gfrec} over the index $i$ one obtains a compact recursion for the generating functions $F_n(X)$.
\begin{prop}
For $1\leq i\leq n+1$, denote 
$F_{n, \geq i}(X) = F_{n,i}(X)+\ldots +F_{n,n+1}(X)$. Then
\begin{equation}
 F_n(X) = \sum_{i=1}^n (x_1+\ldots+x_i) F_{n-1, \geq i} (X\smallsetminus x_i)    
\end{equation}

\end{prop}

\section{Parking functions and asymmetric multinomials}\label{sec:parking-functions}

In this section we give a combinatorial interpretation of $\multichoose{n}{\bf k}$.

\begin{defi}
	For a label $a$ in a unit square of a parking function $P$, define its \defn{dominance index}, written $d_P(a)$, to be the number of columns to its left that contain no label greater than $a$.  
\end{defi}

\begin{defi}
A parking function $P$ is \defn{column-restricted} if for every label $a$, $$d_P(a)<a.$$ We write $\CPF(n,\mathbf{k})$ to denote the set of all column-restricted parking functions having columns of lengths $k_1,k_2,\ldots,k_n$ from left to right.
\end{defi}

An example is shown in Figure \ref{fig:CPF}.

\begin{figure}[b]
	\begin{center}
		\includegraphics{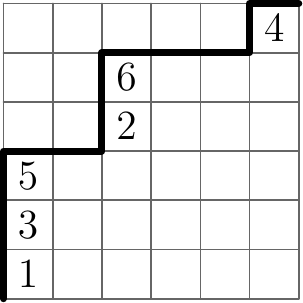}
	\end{center}
	\caption{\label{fig:CPF} A column-restricted parking function $P$.  We have $d_P(1)=d_{P}(2)=d_{P}(5)=0$, $d_P(2)=1$, $d_P(6)=2$, and $d_P(4)=3$.}
\end{figure}

\begin{thm}\label{thm:CPF}
	We have $$|\CPF(n,\mathbf{k})|=\multichoose{n}{\mathbf{k}}$$ for any composition $\mathbf{k}\in \Comp(n,n)$.
\end{thm} 

\begin{proof}
  We show that $|\CPF(n,\mathbf{k})|$ satisfies the recursion of Definition \ref{def:asym}.  First, $|\CPF(1,1)|$ is clearly $1$, since there is only one Dyck path from $(0,0)$ to $(1,1)$ and only one column in which to put the $1$. 
  
  For the recursion, define $\CPF(n,\mathbf{k},j)$ to be the subset of $\CPF(n,\mathbf{k})$ having the $1$ in the $j$-th column. Note that the condition of being column-restricted means that $j<i$ where $k_i$ is the leftmost $0$ in $\mathbf{k}$.  Thus $$\CPF(n,\mathbf{k})=\bigsqcup_{j=1}^{i-1} \CPF(n,\mathbf{k},j).$$  
  It therefore suffices to prove that $\CPF(n,\mathbf{k},j)$ is in bijection with $\CPF(n-1,\widetilde{\mathbf{k}}_j)$ for all $j<i$.
  
 We define a bijection $\varphi:\CPF(n,\mathbf{k},j)\to \CPF(n-1,\widetilde{\mathbf{k}}_j)$ as follows.  For $P\in \CPF(n,\mathbf{k},j)$, define $\varphi(P)$ by removing the row containing $1$ in $P$, then removing the first empty column in the resulting diagram (which may be the column that contained $1$ if that column is now empty), and finally decrementing all remaining labels by $1$.  (See Figure \ref{fig:phi}.)
 
\begin{figure}
	\begin{center}
		\includegraphics{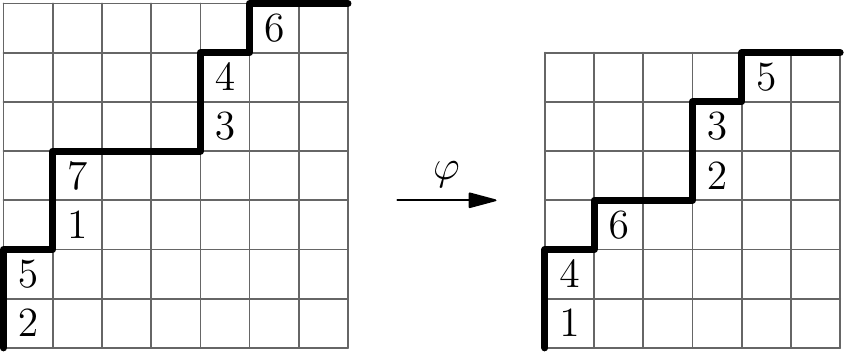}
	\end{center}
	\caption{\label{fig:phi} An example of the map $\varphi$.}
\end{figure}

 Notice that the sequence of column heights of $\varphi(P)$ is $\widetilde{\mathbf{k}}_j$, which is Catalan by Lemma \ref{lem:Catalan}.  It follows that $\varphi(P)$ is a parking function of size $n-1$.
 
 We now show that $\varphi(P)$ is column-restricted. Consider a label $a$ in $P$.  
 
 \textbf{Case 1.} First suppose $a$ is to the left of $1$ in $P$.  Since $P$ is column-restricted, there are no empty columns to the left of the $1$, and hence no empty columns to the left of $a$.  Moreover, the only numbers less than $a$ that can appear to the left of $a$ are the $a-2$ numbers $2,3,\ldots,a-1$, so $$\dom_P(a)\le a-2.$$
 After applying $\varphi$, the label $a$ is replaced by $a-1$, and the labels to its left are decreased by $1$, so $$\dom_{\varphi(P)}(a-1)=\dom_P(a)\le a-2<a-1.$$     Thus the column-restricted condition holds for $a-1$ in $\varphi(P)$.
 
 \textbf{Case 2.} Suppose $a$ is weakly to the right of the $1$ in $P$.  If the $1$ is in its own column in $P$, then removing the row and column of the $1$ to form $\varphi$ decreases the dominance index of $a$ by $1$, and hence $$\dom_{\varphi(P)}(a-1)=\dom_{P}(a)-1\le a-2.$$  If instead the $1$ is in a column with other entries, the proof goes through as in Case 1 if $a$ is to the left of the first empty column, and if $a$ is to the right of the first empty column then again its dominance index decreases by $1$ after deleting the empty column, and we are done.  
 
 We have now shown that $\varphi$ is a well-defined map from $\CPF(n,\mathbf{k},j)$ to $\CPF(n-1,\widetilde{\mathbf{k}}_j)$.  To see that it is a bijection, note that, given a parking function $Q$ in $\CPF(n-1,\widetilde{\mathbf{k}}_j)$, we can first increment each entry by $1$ and then insert a $1$ as follows.  If $k_j=1$, we insert a column consisting of the letter $1$ just after the $(j-1)$st column in $Q$.  If $k_j>1$, we insert a new empty column after the $(i-1)$st column in $Q$ (where $i=\min\{t: k_t=0\}$), and insert a $1$ into column $j$.  This reverses $\varphi$.
\end{proof}

\subsection{\texorpdfstring{Counting by $(2n-1)!!$}{Counting by (2n-1)!!}}

This section proves Theorem \ref{thm:main-2}, by establishing that \begin{equation}\label{eqn:double-factorial}
    \sum_{\mathbf{k}\in \Comp(n,n)} \multichoose{n}{\mathbf{k}}=(2n-1)!!
\end{equation}  The left hand side of (\ref{eqn:double-factorial}) simply counts $|\CPF(n)|$ where $\CPF(n)$ is the set of all column-restricted parking functions of height $n$.  We will show that $|\CPF(n)|=(2n-1)!!$ for all $n\ge 1$.  

Since $|\CPF(1)|=1$, it suffices to show that $$|\CPF(n)|=(2n-1)|\CPF(n-1)|$$ for all $n\ge 2$.  	To do so, note that any Dyck path from $(0,0)$ to $(n-1,n-1)$ passes through exactly $2n-1$ lattice points.  We will show that we can ``insert'' a label $n$ at each of these points to construct a column-restricted parking function of height $n$ from one of height $n-1$.  

\begin{defi}
   A \defn{pointed column-restricted parking function} of size $n$ is a pair $(P,p)$ where $P\in \CPF(n)$ and $p$ is one of the $2n+1$ lattice points on its associated Dyck path.  We write $\CPF_{\bullet}(n)$ for the set of all pointed column-restricted parking functions of size $n$.
\end{defi}

With this in mind, we define the following insertion map.

\begin{defi}\label{def:ins}
  For $(P,p)\in \CPF_{\bullet}(n-1)$, define $\ins(P,p)$ as follows.  Let $P_{p\to}$ be the tail of $P$ (both the path and labels) after the point $p$. 
  
  \begin{enumerate}[left=2\parindent]
      \item[\textbf{Step 1.}] Shift $P_{p\to}$ one step up and one step right.  Connect the newly separated paths by an up step followed by a right step, and label the new up step by $n$.
      
      \item[\textbf{Step 2.}] Let $C_1,\ldots,C_t$ be the columns that contain some entry whose dominance index changed upon performing Step 1 above.  Move the column $C_1$ into the rightmost empty column to its left, then move $C_2$ into the rightmost empty column to its left (which may be the column that $C_1$ occupied before), and so on.  
  \end{enumerate}
  
  The result is $\ins(P,p)$.  Figure \ref{fig:ins} gives a detailed example of this algorithm.
\end{defi}

\begin{figure}
	\begin{center}
		\includegraphics{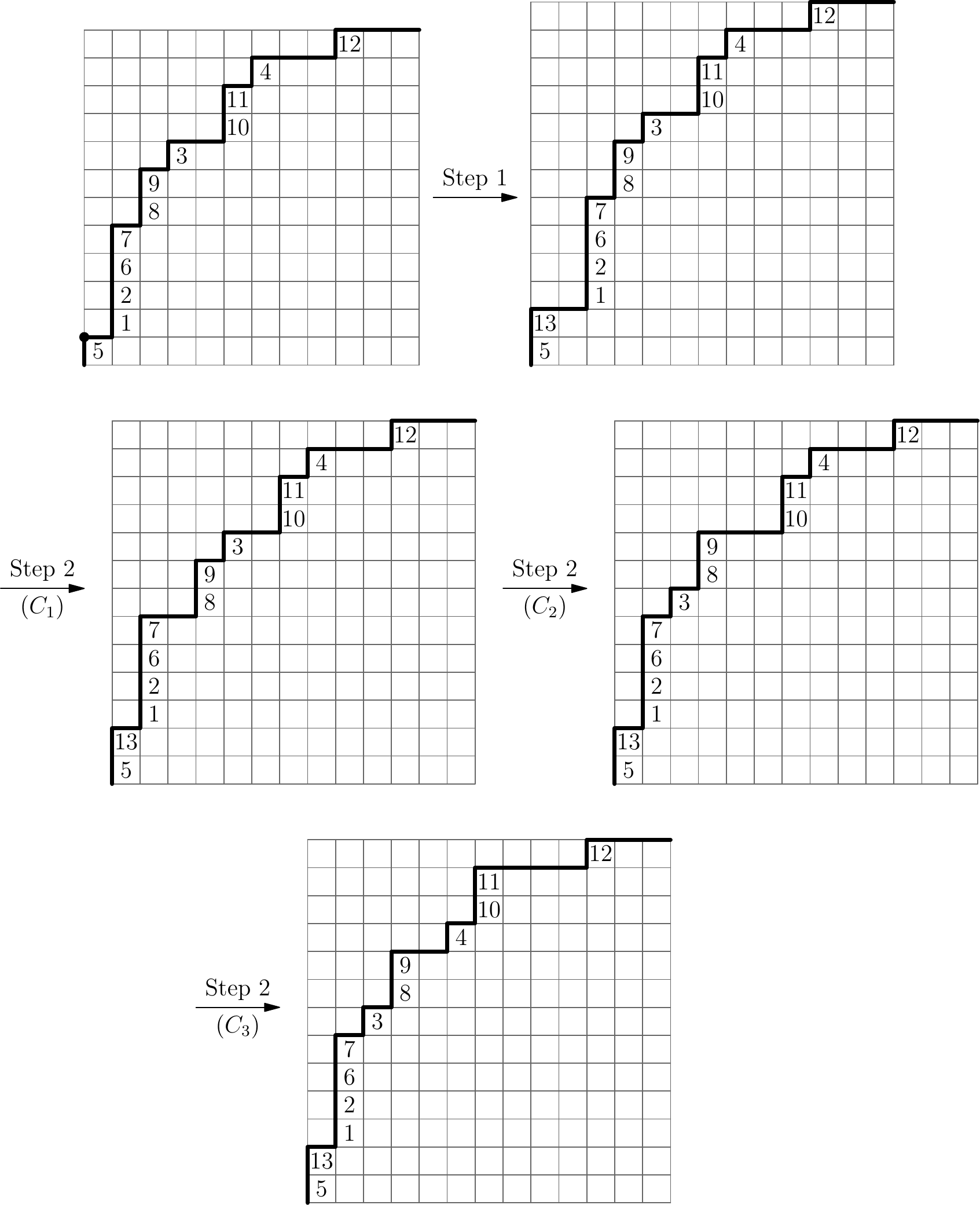}
	\end{center}
	\caption{\label{fig:ins}
		An example of the computation of the map $\ins$, where the dotted corner above the $5$ in the left hand diagram indicates the point $p$ at which we insert $n=13$.  We first perform Step 1 of the algorithm and then break Step 2 down into its individual column moves.}
\end{figure}

We first prove several technical lemmata about the map $\ins$.

\begin{lemma}\label{lem:corners}
    In Step 2 of computing $\ins(P,p)$, we have $t>0$ (i.e., Step 2 is nontrivial) if and only if $p$ is an upper left corner of the Dyck path, that is, it is between an up step and a right step.
    
    Moreover, in this case, let $r$ be the label just below $p$.  Then the labels whose dominance index changes in Step 1 of $\ins$ are precisely those labels $a<r$ to the right of $r$, and their dominance index increases by exactly $1$.
\end{lemma}

\begin{proof}
    First suppose $p$ is not an upper left corner.  Then $p$ is either preceded by a right step or is between two up steps.  In the former case, step $1$ of computing $\ins(P,p)$ simply inserts a new column containing only the entry $n$.  Since all entries in these columns are less than $n$, their dominance index does not change.  In the latter case, if $p$ is between two up steps, the column to its right is split into two columns and the $n$ is inserted at the top of the first half.  Thus the column containing $n$ does not add to the dominance index of any entry to its right, and we are done as before.
    
    Now suppose $p$ is an upper left corner.  Let $r$ be the label just below $p$, at the top of its column.  Then Step 1 inserts $n$ directly above $r$, and adds an empty column to its right.  Let $a$ be a label to the right of $r$.  If $a>r$, its dominance index decreased by $1$ from inserting $n$ above $r$, but increased by $1$ from the addition of the empty column, so its dominance index was unchanged.  If instead $a<r$, then its dominance index simply increases by $1$ via the new empty column.
\end{proof}

Lemma \ref{lem:corners} gives rise to the following natural definitions.

\begin{defi}
   We write $\Noncorner(n-1)$  to denote the pairs $(P,p)\in \CPF_\bullet(n-1)$ in which $p$ is not an upper left corner, and  $\Corner(n-1)$  for pairs where $p$ is an upper-left corner.  We refer to these types as \defn{good} and \defn{bad} pointed CPF's respectively.
\end{defi}

We can also tell from the output of $\ins(P,p)$ whether $(P,p)$ is good or bad.

\begin{lemma}
  We have $(P,p)\in \Noncorner(n-1)$ if and only if, in $\ins(P,p)$, either (a) there is no label below $n$ in its column, or (b) there is a label $r$ below $n$ and the square up-and-right from $n$ contains a label $a>r$.  
\end{lemma}

\begin{proof}
    This follows immediately from the same casework as in Lemma \ref{lem:corners}.
\end{proof}

We therefore may define good and bad (non-pointed) parking functions of size $n$ as well.

\begin{defi}\label{def:GPF-BPF}
A parking function $Q$ in $\CPF(n)$ is \defn{good} if either (a) there is no label below $n$ in its column, or (b) there is a label $r$ below $n$ and the square up-and-right from $n$ contains an entry $c>r$.  If $Q$ is not good, we call it \defn{bad}, and this occurs if and only if the square below $n$ contains a label $r$ and the square up-and-right from $n$ either is empty or contains a label $c$ with $c<r$.

We write $\Good(n)$ and $\Bad(n)$ for the sets of good and bad column-restricted parking functions of height $n$, respectively.
\end{defi}

\begin{example}
The example shown in Figure \ref{fig:ins} starts with a bad pointed parking function $(P,p)$, and the output $\ins(P,p)$ is bad as well.  The two examples shown in Figure \ref{fig:GPF} illustrate the map $\ins$ on good pointed parking functions, and the output $\ins(P,p)$ is good in these cases. 
\end{example}

\begin{figure}
	\begin{center}
		\includegraphics{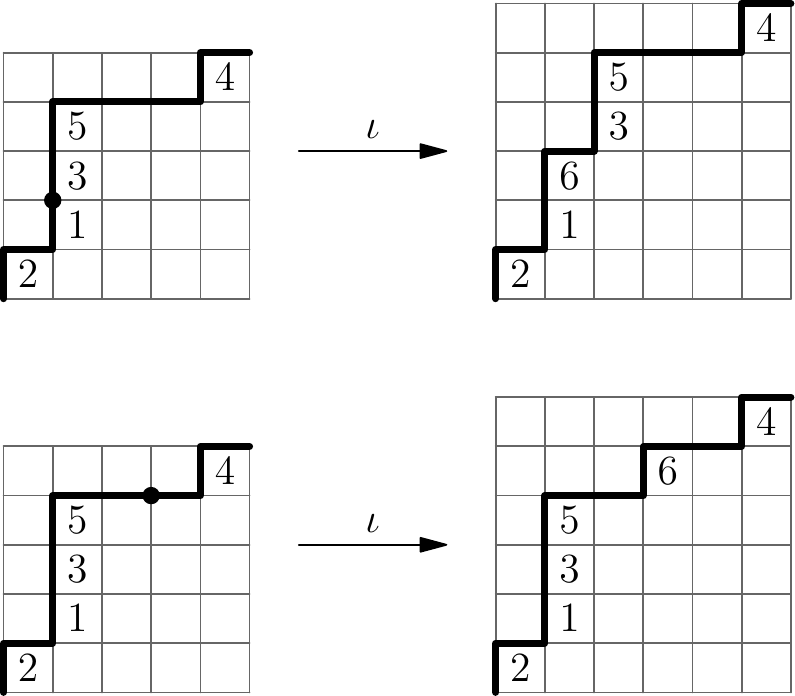}
	\end{center}
	\caption{\label{fig:GPF}Two examples of $\ins$ applied to a good pointed column restricted parking function $(P,p)\in \Noncorner(5)$, where the black dot marks the point $p$ in each of the left hand parking functions above.  Note that only Step 1 applies in each case, as Step 2 is vacuous.}
\end{figure}

\begin{lemma}
	The map $\ins:\CPF_\bullet(n-1)\to \CPF(n)$ is well-defined, and it restricts to maps $$\ins:\Noncorner(n-1)\to \Good(n)$$ and $$\ins:\Corner(n-1)\to \Bad(n).$$
\end{lemma}

\begin{proof}
   To show that $\ins$ is well-defined, it suffices to show that the outputs are column-restricted parking functions.
   
   To show that the output $\ins(P,p)$ is a parking function, we need to show that the resulting sequence of column heights is still a Catalan composition.  Let $\mathbf{k}=(k_1,\ldots,k_{n-1})$ be the sequence of column heights of $P$.  
   
   The sequence of column heights after performing Step 1 of $\ins(P,p)$ is formed by splitting some $k_i$ into two (possibly empty) parts $k_i'$ and $k_i''$, and increasing the first part by $1$.  The resulting partial sums $k_1+\cdots+k_t$ are unchanged for $t< i$, and so in particular $k_1+\cdots +k_{i-1}\ge i-1$ so $k_1+\cdots+k_{i-1}+(k'_i+1)\ge i$.  For $t>i$, we have that the $t$-th partial sum of the new sequence is $$k_1+\cdots+k_{i-1}+(k'_i+1)+(k''_i)+k_{i+1}+\cdots+k_{t-1}=k_1+\cdots+k_{t-1}+1\ge t-1+1=t$$ and so the new sequence of column heights is Catalan.
   
   It follows that, if $(P,p)$ is good (and hence there is no Step 2), $\ins(P,p)$ is a parking function.  Moreover, by Lemma \ref{lem:corners}, $\ins(P,p)$ is column restricted since the dominance indices of each entry do not change.  Thus $\ins:\Noncorner(n-1)\to \Good(n)$ is well-defined.
   
   Now suppose $(P,p)$ is bad.  Then Step 2 of computing $\ins(P,p)$ simply moves some columns to the left, so this only increases the partial sums and the resulting column heights sequence is still Catalan. Thus $\ins(P,p)$ is a parking function.  To see that it is column restricted, let $r$ be the entry just below the corner $p$ as in Lemma \ref{lem:corners}.  The dominance index of the entries weakly left of $r$ do not change.  For the entries to the right of $r$, if $a<r$ then its column is moved one step left into an empty column, which decreases its increased dominance index by $1$ and hence we still have $\dom_P(a)<a$ after Step 2.  If $a>r$ then moving columns to the left can only decrease its dominance index.  Thus $\ins(P,p)$ is column-restricted.
\end{proof}

We now show that the maps $\ins:\Noncorner(n-1)\to \Good(n)$ and $\ins:\Corner(n-1)\to \Bad(n)$ are bijections.  We define the inverse map as follows.

\begin{defi}
  Define $\rem:\CPF(n)
   \to \CPF_\bullet(n-1)$ via the following two-step algorithm.  For any $Q\in \CPF(n)$:  
   \begin{enumerate}
       \item If $Q$ is bad, let $r$ be the entry below $n$ in $Q$.  Let $C_1,\ldots,C_t$ be the columns to the right of $r$ containing an entry $a<r$.  Move $C_t$ into the nearest empty column to its right, and then move $C_{t-1}$ in the same manner, and so on.
       \item If $Q$ is good, or if it is bad and we have just performed Step 1 above, then set $p$ to be the lattice point in the lower left corner of the square containing $n$, remove $n$ from its column, and shift the tail of the path after $n$ one step down and one step left. 
   \end{enumerate}  
   Then if $P$ is the resulting parking function, define $\rem(Q)=(P,p)$.
\end{defi}

If $\rem$ is well-defined, then it is an inverse of $\ins$.  The following lemma therefore completes the proof.

\begin{lemma}
	The map $\rem:\CPF(n)\to \CPF_\bullet(n-1)$ is well-defined, and it restricts to maps $$\rem:\Good(n)\to \Noncorner(n-1)$$ and $$\rem:\Bad(n)\to \Corner(n-1).$$
\end{lemma}

\begin{proof}
  Let $Q\in \CPF(n)$.  If $Q$ is good, then $\rem(Q)=(P,p)$ is formed by removing the $n$ and shifting all later columns one step left (merging the column that contained $n$ with the next column, which always results in a valid column having increasing entries by the definition of good).  Since the partial sums of the column heights decrease by $1$ but the indices also decrease by $1$, the sequence of column heights in $P$ is Catalan.  Moreover, all entries retain their dominance index from $Q$ to $P$.  Finally, by the definition of good, $p$ is not an upper left corner of the diagram.  Thus if $Q\in \Good(n)$ then $\rem(Q)\in \Noncorner(n-1)$.
  
  Now suppose $Q\in \Bad(n)$.  Let $r$ be the entry below $n$ in $Q$ and let $C_1,\ldots,C_t$ be the columns (listed from left to right) to the right of $r$ containing some entry $a<r$.  Then $\rem(Q)=(P,p)$ is formed by first shifting the columns $C_t,C_{t-1},\ldots,C_1$ in that order to the nearest empty columns to their right, and then removing the $n$ and shifting all columns to the right of it one step left.  
  
  We first show that the sequence of column heights of $P$ remains Catalan.  To do so, we must show that the column heights are still Catalan after shifting each of $C_t,\ldots,C_1$ to the right, since the last step of removing the $n$ and shifting left does not change the Catalan property (as in the good case).  Notice that moving a column $C$ into the first empty column to its right retains the Catalan property if and only if the bottom entry of column $C$ was \textit{strictly} above the diagonal to begin with.  So, we simply need to show that any element $b<r$ to the right of $r$ in $Q$ lies strictly above the diagonal.
  
  Let $b$ be such an entry in $Q$, and let $e$ be the number of empty columns to the left of $b$ and $s$ the number of nonempty columns to the left of $b$ (including the column containing $r$ and $n$).  Let $j$ be the number of the nonempty columns whose largest entry is less than $b$, and denote the largest entries of these columns $b_1,\ldots,b_j$ where $b_1<b_2<\cdots<b_j$.  
  
  \textbf{Claim.} At least $e+j$ of the numbers in $\{1,2,\ldots,b-1\}$ are to the left of $b$ in $Q$.

  To prove this claim, note that since there are $e<b$ empty columns, the numbers $1,2,\ldots,e$ must be to the left of $b$ in $Q$, for otherwise their dominance index would be too high (since $Q$ is column-restricted).  Moreover, suppose exactly $j_0$ of the numbers $b_1,\ldots,b_j$ are less than $e$, and $j_1=j-j_0$ are greater, so that $$b_1<\cdots<b_{j_0}<e<b_{j_0+1}<\cdots<b_{j}.$$
  Then the $j_1$ entries $b_{j_0+1}<\cdots<b_j$ are left of $b$ by assumption.  But since $j_0$ of the largest entries of the columns to the left of $b$ are less than $e$, the smallest $j_0$ letters among $$\{e+1,\ldots,b-1\}\smallsetminus \{b_{j_0+1},\ldots,b_j\}$$ cannot be to the right of $b$ either, for otherwise their dominance index would be too large.  It follows that there are at least $e+j_1+j_0=e+j$ entries among $\{1,2,\ldots,b-1\}$ to the left of $b$, proving the claim.
  
  Note that, to the left of $b$, there are $s-j$ columns having largest entry greater than $b$, $e+j$ entries less than $b$, and the entry $r$. Thus there are at least $(e+j)+(s-j)+1=e+s+1$ distinct entries to the left of $b$.  Since there is one entry per row in any parking function, the number of rows below $b$ is greater than $e+s$, and $e+s$ is the number of columns to the left of $b$ by the definition of $e$ and $s$.  It follows that $b$ lies strictly above the diagonal, as desired.
  
  We have now shown that $P$ is a parking function, and it remains to show that it is column-restricted.  The entries weakly left of $r$ are unchanged from $Q$ to $P$, so we consider the entries to the right of $r$ in $Q$.  
  
  Suppose $b<r$ is to the right of $r$.  Then the column containing $b$ will be moved to the right past some number of consecutive columns whose smallest entry is greater than $r$.  This increases the dominance index of $b$ by $1$, but then removing the $n$ and shifting the columns to the left decreases its dominance index by $1$. Since $Q$ is column-restricted, we have $\dom_P(b)=\dom_Q(b)\le b-1$.
  
  Now consider an entry $a>r$ to the right of $r$ that is in a column $C$ that is moved to the right, so that there is an entry $b<r$ in $C$ as well.  Let $i$ be the number of columns that $C$ moves past whose largest entry is less than $a$ (and necessarily greater than $b$).  Then since $\dom_Q(b)\le b-1$, and the number of columns to the left of $a$ whose largest entry is between $b+1$ and $a-1$ is at most $|\{b+1,\ldots,a-1\}|-i=a-1-b-i$, we have that $$\dom_Q(a)\le b-1+a-b-1-i=a-2-i.$$  After moving $C$ to the right, the dominance index of $a$ increases by exactly $i+1$, and removing the $n$ and shifting the columns left does not affect the dominance index since $a>r$.  Thus $$\dom_P(a)=\dom_Q(a)-(i+1)\le a-2-i-(i+1)=a-1$$ and so the column restricted condition holds at $a$.
  
  Finally, consider an entry $a'>r$ to the right of $r$ that does not move.  If no column $C$ moves past $a'$ then its dominance index does not change.  Otherwise, suppose a column $C$ whose largest entry is $a$ moves past the column containing $a'$ in forming $\rem(Q)$.  If $a<a'$ then the dominance index of $a'$ does not change, so suppose $a>a'$.  Then since $\dom_Q(b)\le b-1$ and there are no empty columns between $C$ and $a'$, we have $$\dom_Q(a')\le (b-1)+(a-1-b)=a-2.$$  Since moving $C$ past $a'$ increases the dominance index of $a'$ by exactly $1$, we have $\dom_P(a')\le a-1$ as desired.
\end{proof}

It follows that the map $\ins:\CPF_\bullet(n-1)\to \CPF(n)$ is a bijection, and equation (\ref{eqn:double-factorial}) follows.  This completes the proof of Theorem \ref{thm:main-2}.

\subsection{An alternative insertion algorithm}

In the previous section, we established equation (\ref{eqn:double-factorial}) by algorithmically defining a bijection $\ins:\CPF_{\bullet}(n-1)\to \CPF(n)$.  We now define a different bijection $$\ins':\CPF_{\bullet}(n-1)\to\CPF(n)$$ that achieves the same result. 

\begin{remark}\label{rmk:not-as-good}
  While the map $\ins$ preserves the partition of $[n]$ into columns (though may reorder the columns), the map $\ins'$ does not.  However, as we shall see below, the bijectivity of $\ins'$ has the advantage of having a much simpler proof than that of $\ins$.  For this reason we include both bijections in this discussion.
\end{remark}

\begin{defi}
  For an element $(P,p)\in \CPF_\bullet(n-1)$, we define $\ins'(P,p)$ as follows.   Let $P_{p\to}$ be the tail of $P$ after $p$ as defined in Definition \ref{def:ins}. 
  
  \textbf{Case 1:} Suppose $p$ is not an upper-left corner of the Dyck path of $P$.  Shift $P_{p\to}$ one step up and one step right.  Connect the newly separated paths by an up step followed by a right step, and label the new up-step by $n$.  The result is $\ins(P,t)$.
  
  \textbf{Case 2:} Suppose $p$ is an upper-left corner of $D$.  Shift $P_{p\to}$ one step up, connecting the paths with an up-step and giving it the label $n$.  Let $r$ be the highest label below $p$, that is, the label just below $n$.  Then, for each label $a$ in $P_{p\to}$ in order from top to bottom, perform the following action based on the three subcases below.  
  \begin{enumerate}
  	\item[(a)] If $a<r$, do nothing and proceed to the next label below $a$.  	
  	\item[(b)] If $a>r$ and if moving $a$ one square to the right results in all increasing columns, do so.  Proceed to the next label below $a$.
  	\item[(c)] If $a>r$ but we cannot move $a$ one square to the right, let $c_1<\cdots<c_u$ be the labels in the column just to the right of $a$ that are less than $r$, and let $b_1<\cdots<b_v$ be the labels in the column of $a$ that are less than $r$.  Then we interchange the sets of numbers $\{c_i\}$ and $\{b_i\}$ between the two columns.  Proceed with the next label \textit{strictly to the left} of $a$. 
  	\end{enumerate}   
\end{defi}

We illustrate the map $\ins'$ in Figure \ref{fig:ins-prime}.

\begin{figure}
	\begin{center}
		\includegraphics{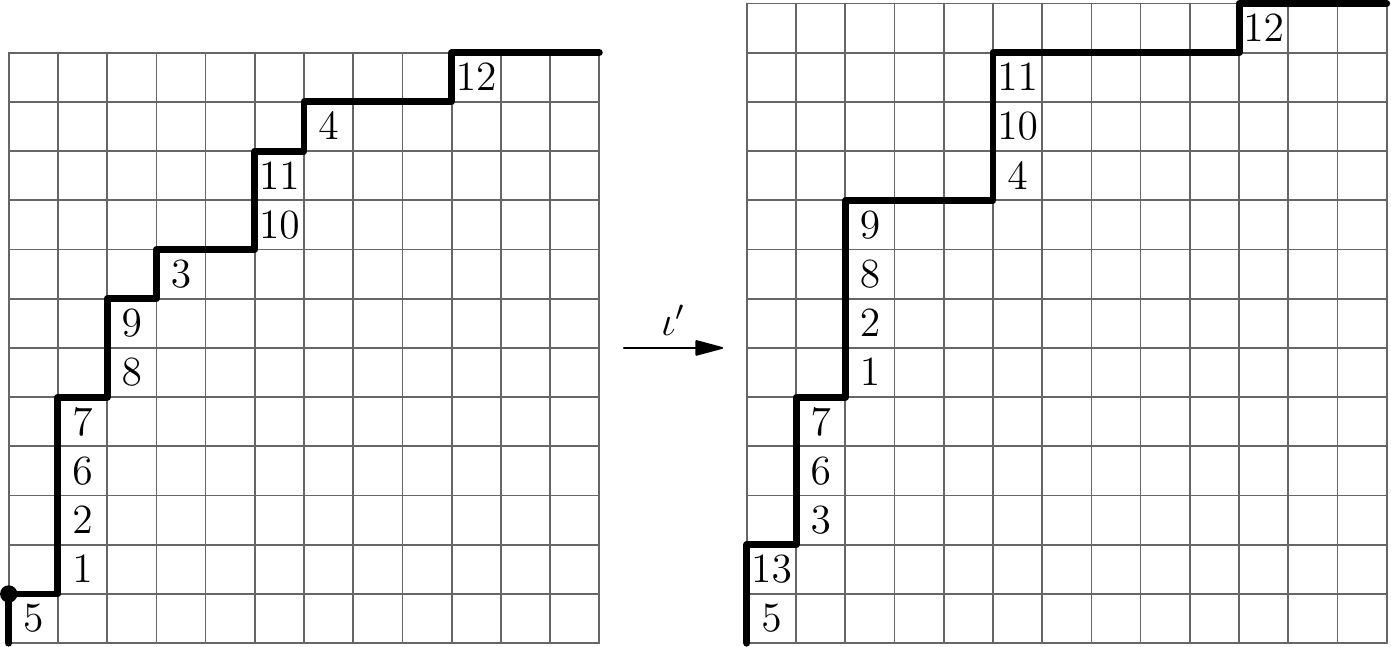}
		
	\end{center}
	\caption{\label{fig:ins-prime}   An example of the map $\ins'$, where the dotted corner in the left hand diagram indicates the point $p$ at which we insert $n=13$.}
\end{figure}

Notice that, in Case 2 of the algorithm for $\ins'$, the numbers $a>r$ are precisely those whose dominance index changes upon inserting the $n$ and shifting $P_{p\to}$ up one step, and in particular their dominance index decreases by $1$.  Shifting them to the right restores their original dominance index when possible (Case 2(b)).  

 We now provide two important lemmata about the map $\ins'$.

\begin{lemma}\label{lem:ins-structure}
   In Case 2(c), there must exist a nonempty collection of entries $c_1<\cdots <c_u$ in the column to the right of $a$ that are less than $r$.  
\end{lemma}

\begin{proof} 
This follows by a  strong induction argument on the steps in Case 2 of the computation of $\ins'$.
\end{proof}

\begin{lemma}
	The map $\ins'$ is a well-defined map from $\CPF_\bullet(n-1)$ to $\CPF(n)$.
\end{lemma}

\begin{proof}
	 Let $(P,p)\in \CPF_\bullet(n-1)$.  If the computation of $\ins'(P,p)$ is of the type in Case 1, an argument identical to that of $\ins$ shows that $\ins'(P,p)\in \CPF(n)$.
	 
	 In Case 2, note that all labels $a>r$ in $P_{p\to}$ were first moved up one step and then possibly to the right one step, so these all still lie weakly above the diagonal.  The entries $a<r$ that are moved via Case 2(c) are moved from one column that starts above the diagonal to an adjacent column or vice versa, and both columns therefore stay above the diagonal as well.  Thus the path of $\ins'(P,p)$ is still a Dyck path.
	 
	 Additionally, the dominance index of each entry $a<r$ does not change, and if $a>r$ it can change by $-1$ if $a$ does not move and by $0$ if it does move to the right.  In either case the new parking function is still column-restricted, so $\ins(P,t)\in \CPF(n)$.
\end{proof}

\begin{thm}
	The map $\ins':\CPF_\bullet(n-1)\to \CPF(n)$ is a bijection, and it restricts to bijections $$\ins':\Noncorner(n-1)\to \Good(n)$$ and $$\ins':\Corner(n-1)\to \Bad(n).$$
\end{thm}

\begin{proof}
	Note that $\ins'$ and $\ins$ are the same function on $\Noncorner(n-1)$, and so we immediately have that $$\ins':\Noncorner(n-1)\to \Good(n)$$ is a bijection.
	
    Now let $(P,p)\in \Corner(n-1)$.  Then by Lemma \ref{lem:ins-structure} and the definition of a bad parking function (Definition \ref{def:GPF-BPF}), we see that $\ins'(P,p)\in \Bad(n)$.  
    
    To show that the restriction $\ins':\Corner(n-1)\to \Bad(n)$ is a bijection, let $Q\in\Bad(n)$.  Let $r$ be the label in $Q$ just below $n$ in its column, which exists by the definition of $\Bad(n)$.   Let $Q_{\to}$ be the tail in $Q$ after the upper left corner of the square containing $n$.  We define $\rem'(Q)$ as follows.  Remove the $n$ and its adjacent up-step and shift $Q_{\to}$ down one step.  Then, perform the following action on each label $a$ in $Q_{\to}$ in order from bottom to top:
 \begin{enumerate}
 	\item[(a)] If $a<r$, do nothing and proceed to the next label above $a$.  	
 	\item[(b)] If $a>r$ and if moving $a$ one square to the left results in all increasing columns, do so and proceed to the next label above $a$.
 	
 	\item[(c)] If $a>r$ but we cannot move $a$ to the left, let $c_1<\cdots<c_u$ be the labels less than $r$ that occur below $a$ in its column.   Also let $b_1<\cdots<b_v$ be the labels in the next column to the right of $a$ that are less than $r$.  Then we interchange the sets of numbers $\{c_i\}$ and $\{b_i\}$ between the two columns.  Finally, resume this process starting with the first label above the new position of $c_u$. 
 \end{enumerate}   
We set $P$ to be the resulting parking function, and set $p$ to be the northwest corner of the label $r$ in $P$.  Then we define $\rem(Q)=(P,p)$.  Note that the condition of $Q$ being bad implies that $\rem(Q)\in \Corner(n-1)$.

Furthermore, if we are in Case (c) above, a similar strong induction argument as in Lemma \ref{lem:ins-structure} shows that the set $\{c_i\}$ must be nonempty at such a step.  Thus, after interchanging $\{c_i\}$ and $\{b_i\}$, we end up with a number $a>r$ that cannot be moved one step to the right, matching Case 2(c) of the definition of $\ins'$.  It now follows that $\rem'$ is an inverse of $\ins'$ on $\Corner(n-1)$, and so $\ins'$ is a bijection. 
\end{proof}

\bibliography{myrefs}
\bibliographystyle{plain}
\end{document}